\documentclass{ws-sd}

\usepackage{latexsym}
\usepackage{graphicx}
\usepackage{amsmath}
\usepackage{amsfonts}
\usepackage{amssymb}
\usepackage{mathrsfs}
\usepackage{nicefrac}
\usepackage[T1]{fontenc}
\usepackage{siunitx}
\usepackage[ansinew]{inputenc}
\usepackage{natbib}
\bibpunct{(}{)}{;}{a}{,}{,}

\newcommand{\ud}[1]{\, \mathrm{d}#1}
\newcommand{\deriv}[3][]{\frac{\ud^{#1} \hspace{-0.3mm} #2}{\ud{#3}^{#1}}}

\newcommand{\grad}{\nabla }
\renewcommand{\P}{ \mathbb P }
\newcommand{\EXP}{\mathbb{E}}

\newcommand{\eps}{\epsilon}

\newcommand{\R}{\mathbb{R}}
\newcommand{\A}{\mathcal{A}}
\newcommand{\Dom}{\text{Dom}}

\renewcommand{\geq}{\geqslant}
\renewcommand{\leq}{\leqslant}
\renewcommand{\d}{\partial}
\newcommand{\tr}{\intercal}
\newcommand{\bs}{\boldsymbol}
\newcommand{\fP}{f_{\Pb}}
\newcommand{\fY}{f_{\Yb_{\! 2}}}
\newcommand{\fYt}{\tilde{f}_{\Yb_{\! 2}}}
\newcommand{\fPt}{\tilde{f}_{\Pb}}
\newcommand{\Qb}{\boldsymbol Q}
\newcommand{\xb}{\boldsymbol x}
\newcommand{\ab}{\boldsymbol a}
\newcommand{\vpb}{\boldsymbol \varphi}

\newcommand{\yb}{\boldsymbol y}
\renewcommand{\sb}{\boldsymbol s}

\newcommand{\zb}{\boldsymbol z}
\newcommand{\Xb}{\boldsymbol X}
\newcommand{\Yb}{\boldsymbol Y}
\newcommand{\Rb}{\boldsymbol R}
\newcommand{\Pb}{\boldsymbol P}
\newcommand{\mb}{\boldsymbol{m}}

\newcommand{\Momr}[1]{M^{(#1)}_{Q_r}}
\newcommand{\MomP}[1]{M^{(#1)}_{P}}
\newcommand{\edge}[1]{e\oldstylenums{#1}}
\newcommand{\Rnn}{\R_+^{2n_{\Gamma}}}
\newcommand{\Rn}{\R_+^{n_{\Gamma}}}

\newcommand{\AG}{\boldsymbol{\Lambda}_\Gamma}

\newcommand{\Mb}{\mathbf{M}}
\newcommand{\Kb}{\mathbf{K}}

\newcommand{\Hb}{\mathbf{H}}
\newcommand{\Ib}{\mathbf{I}}

\newcommand{\gt}{\tilde{g}}

\begin{document}

\markboth{Authors' Names}{Paper's Title}

\catchline{}{}{}{}{}

\title{DYNAMICS OF DRAINAGE UNDER STOCHASTIC RAINFALL IN RIVER NETWORKS }

\author{JORGE M RAMIREZ }

\address{Escuela de Matem\'{a}ticas, Universidad Nacional de Colombia,\\
Calle 59A No 63-20, Medell\'{i}n, Colombia, 050034,\\ jmramirezo@unal.edu.co}

\author{CORINA CONSTANTINESCU}

\address{Mathematical Sciences, University of Liverpool, L69 3 BX\\
Liverpool, United Kingdom\\
C.Constantinescu@liverpool.ac.uk}

\maketitle

\begin{history}
\received{(Day Month Year)}
\revised{(Day Month Year)}
\end{history}

\begin{abstract}
We consider a linearized dynamical system modelling the flow rate of water along the rivers and hillslopes of an arbitrary watershed. The system is perturbed by a random rainfall in the form of a compound Poisson process. The model describes the  evolution, at daily time scales, of an interconnected network of linear reservoirs and takes into account the differences in flow celerity between hillslopes and streams as well as their spatial variation. The resulting stochastic process is a piece-wise deterministic Markov process of the Orstein-Uhlembeck type. We provide an explicit formula for the Laplace transform of the invariant density of streamflow in terms of the geophysical parameters of the river network and the statistical properties of the precipitation field. As an application, we include novel formulas for the invariant moments of the streamflow at the watershed's outlet, as well as the asymptotic behavior of extreme discharge events.
\end{abstract}



\section{Introduction}\label{Sec_Intro}	

River networks are a chief example of interconnected dynamical systems operating under stochastic forcing. The emerging properties of these systems define the process by which rainfall is converted into river discharge through the accumulation of hillsolpe runoff. Therein lies the fundamental problem of hydrology. Uncertainty plays a key role too, and the pronounced temporal variability of runoff and discharge reflects the random character of key hydrologic fluxes \citep{botter2007probabilistic}.

We propose and solve a stochastic differential model for the streamflow and subsurface runoff throughout an arbitrary watershed under a random precipitation field. The proposed conceptual model uses the linearized equations for conservation of mass and momentum on each river and hillslope proposed by  \cite{gupta1998spatial, rodriguez1999probabilistic}, and aggregates them according to the geometry of the river network. The focus of this paper is to derive the necessary mathematical details leading to the solution of the equations and some interesting initial consequences. Our results include estimates on invariant moments of discharge and asymptotics of extreme events.

The time scales of interest are of days or longer, hence precipitation events are assumed instantaneous. The spatial scale is arbitrary, with individual stream links and hillslopes considered as interconnected linear reservoirs, whereas the detailed dynamics of soil moisture, infiltration or evapotranspiration are neglected. The proposed model does take into account the kinematic delay among the contributions originating from different sub-basins. Moreover, hydraulic parameters are specified in the model at individual stream links and hillslopes, effectively incorporating variations in celerity due to location or scale.

Among the calculations presented here are explicit expressions for the Laplace transform of the densities of both, the transition probabilities and the unique invariant distribution of the process $ \Xb $. See Proposition \ref{Prop_MainResult} below. These expressions completely characterize the distribution of the runoff and discharge within the basin for all times, as well as its behavior as $ t \to \infty $. They also explicitly show how the uncertainty associated with the precipitation interacts with the geometry of the river network, and is converted into the uncertainty of discharge and runoff.

As an application, we obtain formulas for the $ n $-th moment of the streamflow at the basin outlet, and the asymptotic behavior of the probabilities of extreme discharge events. The analysis also yields a novel family of geomorphological coefficients that completely characterizes the invariant distribution of $ \Xb $.

\section{Description of the model}\label{Sec_fullModel}

Natural river networks can be modeled as finite directed rooted binary trees \cite{kovchegov2018random}. Let $ \Gamma $ denote such a tree  modeling a river network as in Figure \ref{Fig_Basin}. The edges of $ \Gamma $ are called `links', there are $ n_\Gamma $ of them, and each is denoted in general by the letter $ e $. The most downstream edge, or root, is always denoted by $ r $. The hillslope area that drains through the downstream end of link $ e $ is denoted by $ a_e $. We denote the vector of areas by $ \ab $ and the total watershed area by $ a $: 
\begin{equation}\label{Def_Areas}
\ab := [a_e: e \in \Gamma]^{\tr}, \;\; a := \sum_{e \in \Gamma} a_e.
\end{equation}  
For any time $ t \geq 0 $, the quantities of interest are: the total subsurface runoff $ R_e(t) $ from the hillslopes into the link $ e $, and the streamflow $ Q_e(t)$ at its downstream end, both in units of volume per unit time.

\begin{figure}[h] \centering
	\includegraphics[scale=0.7]{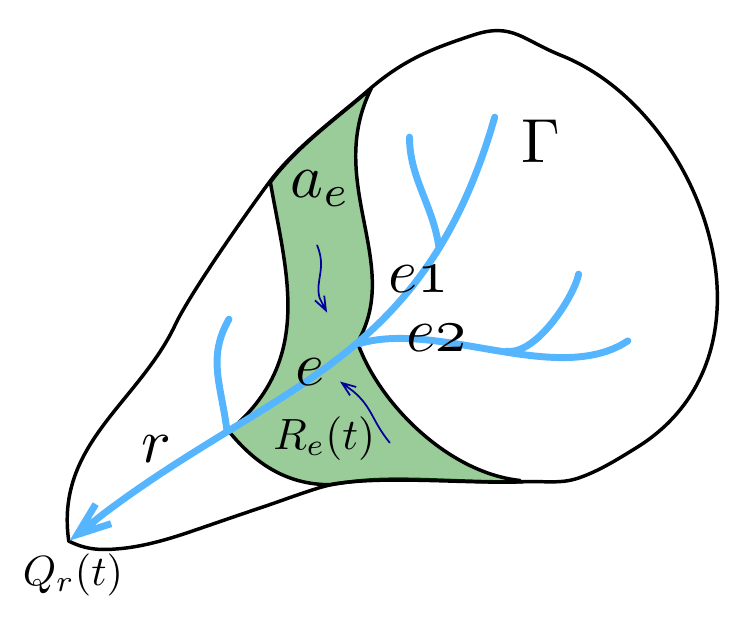}
	\caption{\small{
			Schematic representation of a river network $ \Gamma $ and its components with $ n_{\Gamma} = 9 $ streams.
	}} \label{Fig_Basin}
\end{figure}

Our main dynamical assumption rests on the linearized stream-based conservation equations proposed in \citet{gupta1998spatial} and the linearized subsurface water storage equation of \cite{rodriguez1999probabilistic}. Namely, if the link $ e $ has tributaries $ \edge{1} $ and $ \edge{2}$, then
\begin{align}
\deriv{Q_e}{t}  &= K_e (-Q_e + Q_{\edge{1}} + Q_{\edge{2}} + R_e), \label{Eq_ODEsQ}\\
\deriv{R_e}{t} &= - H_e \left( R_e + a_e \deriv{P_e}{t} \right) \label{Eq_ODEsR}
\end{align} 
The \textit{inverse residence times} $ K_e $ and $ H_e $ play a pivotal role in what follows, and are obtained by supposing that the total storage in stream $ e $ and its hillslope are $ \frac{1}{K_e}Q_e(t) $ and $ \frac{1}{H_e}R_e(t) $ respectively. Typical values of $ H_e $ and $ K_e $ may be deduced from related parameterizations in the literature. \cite{rodriguez1999probabilistic} for instance, used $ H_e = \frac{k_s}{n z} $ where $ k_s $ is the hydraulic conductivity of the saturated soil, $ n $ is the porosity and $ z $ the soil depth. In a more recent work, \cite{rupp2006use} propose a power-law model for $ R $ which, in the limit of homogeneous hydraulic conductivity, can be linearized to equation \ref{Eq_ODEsR} with $ H_e \approx \frac{k_s L_e \sin(i)}{n a_e}$. Here $ L_e $ is the length of the stream and $ \tan(i) $ is the slope of the hillslope. A common approximation to $ K_e $ is given for the classical Mukingum method of flood routing as $ K_e \approx 1.5 v/L_e $, where $ v $ is the average velocity in the channel \cite{dooge1973linear}. See also \cite{Mantilla:2011co}. Using typical values, one arrives at the following ranges for the non-dimensional quantities of interest \cite[see also][Figure 4]{botter2007basin}:
\begin{equation}
\frac{H_e}{K_e} \sim 10^{-3} - 10^{0}, \quad \frac{\lambda}{H_e} \sim 10^{-3} - 10^{0}.
\end{equation}

The stochastic process of interest will be denoted by $ \Xb $ and is obtained by concatenating two $ n_{\Gamma} $-dimensional time dependent column vectors: $ \Qb(t) = [Q_e(t): e\in \Gamma]^{\tr} $ containing the total streamflow at the most downstream point of each link, and $ \Rb(t) = [R_e(t): e\in \Gamma]^{\tr} $ with the total subsurface runoff from the hillslopes into each corresponding stream link. The state space of $ \Xb $ is therefore $ \Rnn := (0,\infty)^{2n_{\Gamma}} $. We thus write
\begin{equation}\label{Def_X}
\Xb(t) := \left[\begin{array}{c}
\Qb(t)\\ \Rb(t) 
\end{array} \right] \in \R_+^{2 n_\Gamma}, \quad t \geq 0,
\end{equation}  
which exemplifies our notation for vectors in $ \Rnn $ in terms of their $ n_\Gamma $-dimensional sub-vectors,
\begin{equation}\label{Def_Notationxs}
\xb = \left[\begin{array}{c}
\xb_1 \\ \xb_2
\end{array}\right] \in \Rnn, \quad \xb_i = [x_{i,e}: e\in \Gamma]^{\tr} \in \Rn, \quad i=1,2.
\end{equation}   
In addition, we follow the convention of having the first entries of both $ \xb_1 $ and $ \xb_2 $ correspond to the root edge $ r $. In particular, $ X_1(t) = Q_r(t) $ and $ X_{n_\Gamma + 1}(t) = R_r(t) $ in \eqref{Def_X}.

The uncertainty in the model comes solely from the precipitation field $ P_e $ in \eqref{Eq_ODEsR} which, at the daily or larger time scales of interest here, can be approximated by a $ n_{\Gamma} $-dimensional compound Poisson process $ P_e $ of increments $ \{P_{n,e}: n \geq 1\} $ \citep{rodriguez1999probabilistic}. The term $ \ud P_e/\ud t $ in \eqref{Eq_ODEsR} is thus to be understood as a generalized derivative. The basin-wide precipitation process may be written as the following point process
\begin{equation}\label{Def_P}
\Pb(t) := \sum_{n=1}^{N(t)} \Pb_{\! n} \, \delta_{T_n}(t)
\end{equation}  
where the storm times $ \{T_n: n\geq 1\} $ define a Poisson process $ N $ with fixed intensity $ \lambda >0 $,
\begin{equation}\label{Def_Nt}
N(t) := \sup\{n\geq 1: T_n \leq t\}.
\end{equation}
In \eqref{Def_P}, the symbol $ \delta_{T_n} $ denotes the Dirac-delta function concentrated at instant $ T_n $, and $ \{\Pb_n: n \geq 1\} $ is a sequence of random i.i.d vectors 
\begin{equation}\label{Def_Pn}
\Pb_n :=  [P_{n,e}: e \in \Gamma]^{\tr} \in \R_+^{n_\Gamma},
\end{equation}  
each with joint probability density function
\begin{equation}\label{Def_fP}
\fP(\yb) \ud \yb = \P( P_{n,e} \in \ud y_e: e \in \Gamma), \; \yb \in \R_+^{n_\Gamma}.
\end{equation}  

The above formulation concerns the case in which a watershed is subject to a random, yet \textit{statistically stationary precipitation regime}. The constant $ \lambda >0 $ gives the average number of precipitation events per unit time. At time $ T_n $, and independently of everything else, the $ n $-th precipitation event occurs instantaneously dropping a random column of water $ P_{n,e} $ onto the hillslopes surrounding link $ e $. For fixed $ n $, the joint distribution of the vector $[ P_{n,e}, e \in \Gamma]^{\tr} $ is given by $ \fP $. Within the specification of $ \fP $ one can, therefore, include any kind of statistical dependence between the precipitation intensities at different locations throughout the watershed. In particular, one may take the uniform case $ P_{n,e} = P_n $ for all $ e \in \Gamma $ as detailed in Remark \ref{Rem_UnifRain}.

We continue the mathematical formulation by introducing the following matrices 
\begin{equation}
\Kb := \text{ diag}(K_e: e \in \Gamma), \;  \; \Hb := \text{ diag}(H_e: e \in \Gamma), \label{Def_KH}
\end{equation}
and the \textit{incidence matrix}  $ \AG $ of the network $ \Gamma $, which is constructed as follows: $ (\AG)_{e,e} = 1 $ for all links $ e \in \Gamma $ and, if link $ e $ has tributaries $ \edge{1}, \edge{2} $, then 
\begin{equation}\label{Def_IncMatrix}
(\AG)_{e, \edge{1}} = (\AG)_{e, \edge{2}} =-1,
\end{equation}  
with all other entries equal to zero. The matrix $ \AG $ encodes all the topological information of the river network.

The system of equations \eqref{Eq_ODEsQ}-\eqref{Eq_ODEsR} can be written as the following linear stochastic differential equation for $ \Xb $:
\begin{equation}\label{Eq_SDEX}
\ud \Xb(t) = \Mb \Xb(t) \ud t + \ud \Yb(t),
\end{equation}  
where $ \Mb $ is the block matrix 
\begin{equation}\label{Def_M}
\Mb := \left[
\begin{array}{cc}
-\Kb\AG & \Kb\\
\mathbf{O} & -\Hb 
\end{array}
\right] \in \R^{n_\Gamma \times n_\Gamma},
\end{equation}  
and the driving process $ \Yb $ is the $2n_\Gamma $-dimensional compound Poisson process given by
\begin{equation}\label{Def_Yb}
\Yb(t) := \sum_{n=1}^{N(t)} \left[ 
\begin{array}{c}
\mathbf{0}\\ \Hb( \ab \circ \Pb_n)
\end{array}\right].
\end{equation}  
In \eqref{Def_M} and \eqref{Def_Yb}, $ \mathbf{O} $ denotes the $ n_\Gamma \times n_\Gamma $ matrix whose entries are all zero, $ \mathbf{0} $ denotes the $ n_\Gamma $-dimensional zero vector, and for vectors $\xb, \yb \in \R_+^{n_\Gamma}$, 
\begin{equation}
\xb \circ \yb := [x_e y_e: \; e \in \Gamma]^{\tr}
\end{equation}  
denotes the component-wise or Hadamard product operator. 

The defintion of $\Yb $ in \eqref{Def_Yb} makes explicit the assumption that rainfall falls exclusively over the hillslopes and affects  $ \Qb $ only through $ \Rb $. The probability density of each vector in the summation \eqref{Def_Yb} is denoted by: 
\begin{align}
f_{\Yb}(\yb) &= f_{\Yb}(\yb_1,\yb_2) :=  
\begin{cases}
\fY(\yb_2), & \yb_1 = \mathbf{0}\\
0, & \text{otherwise}.
\end{cases}, \; \yb \in \Rnn,  \text{  with} \label{Def_fY}\\
\fY(\xb) &:= \dfrac{1}{\prod_{e \in \Gamma} H_e a_e}\fP\left(\dfrac{\xb}{\Hb \ab} \right), \; \xb \in \Rn \label{Def_fY2},
\end{align} 
with $ \fP $ as in \eqref{Def_fP}. In \eqref{Def_fY2} and below, division of a vector by $ \Hb \ab $ simply  denotes component-wise multiplication by $ [1/(H_e a_e): e \in \Gamma]^{\tr} $.

\section{Analysis of $ \Xb $}
An explicit solution to \eqref{Eq_SDEX} is 
\begin{equation}\label{Eq_SolXt}
\Xb(t) =  e^{\Mb t} \Xb(0)+ \sum_{n=1}^{N(t)}  e^{\Mb(t-T_n)} 
\left[
\begin{array}{c}
\mathbf{0}\\ \Hb( \ab \circ \Pb_n)
\end{array}\right].
\end{equation}  
Namely, $ \Xb $ is a \textit{Piecewise Deterministic Markov Process} (PDMP) \citep{davis1984piecewise}. The sample paths of $ \Xb $ evolve between storm events according to the deterministic map 
\begin{equation}\label{Def_flow}
\vpb(t) := e^{\Mb t}, \quad  t \geq 0,
\end{equation}  
and at times $\{ T_n : n \geq 1 \}$ each component of $ \Rb(t) $ jumps by a random amount at each $ T_n $, while those of $ \Qb(t) $ suffer a discontinuity in the derivative. See Figures \ref{Fig_simul1}a and \ref{Fig_fullOrder4}b. Equivalently, the sample paths of $ \Xb $ can be written as the following stochastic convolution integral:
\begin{equation}\label{Eq_SolXtvp}
\Xb(t) = \int_0^t \vpb(t-s) \ud \Yb_0(s).
\end{equation}   
where we have conveniently defined, $ T_0 := 0 $, $ \ud \Yb_0(0) := \Xb(0) $, $ \ud \Yb_0(s) := \ud \Yb(s) $ for $ s>0 $.

\begin{figure}[h] \centering
	\hspace*{-0.5cm}\includegraphics[scale=0.7]{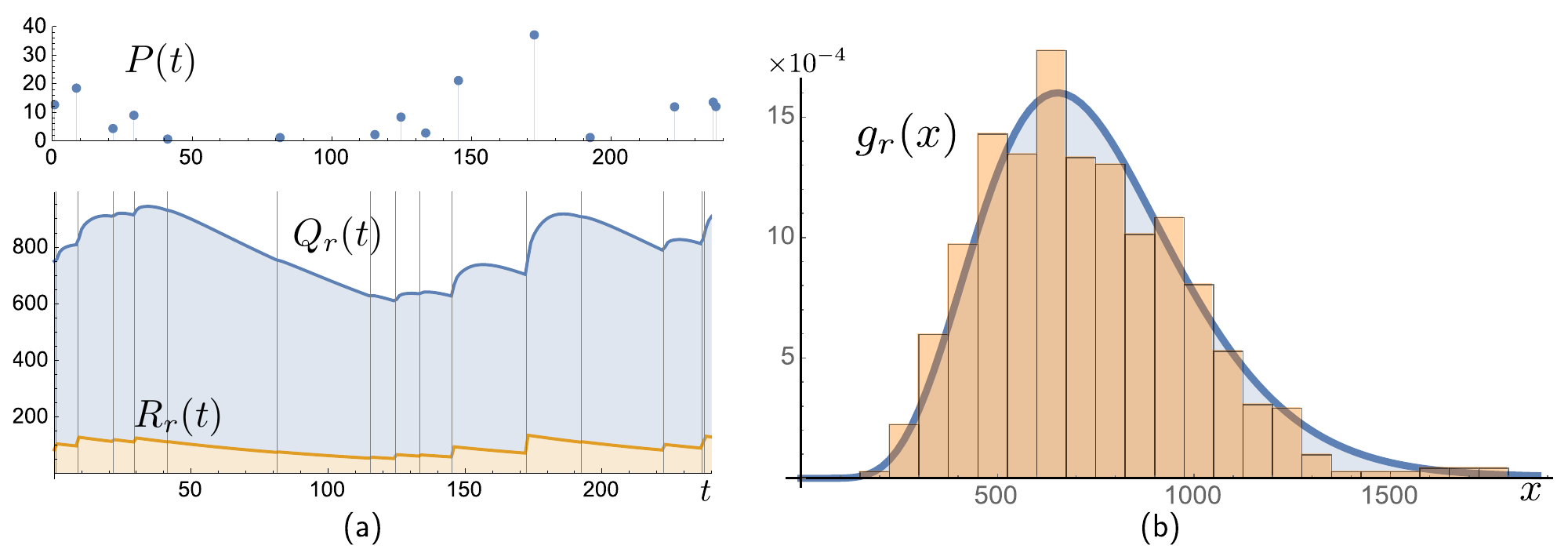}
	\caption{
		{\small
			(a) Simulated $ P(t) $ (\SI{}{\milli\meter}),  $ Q_r(t) $ and $ R_r(t) $ (\SI{}{\liter\per\second}) for the river network in Figure \ref{Fig_Basin}, $ t \in [0,\SI{240}{\hour}]$. All hillslope areas were assumed equal $ a_e = \SI{0.6}{\kilo\meter^2}$.  Values of $ K_e/K_r $ and $ H_e/K_r $ were randomly generated within the intervals $ (0,1) $ and $ (0,10^{-3}) $ respectively, with $ K_r =\SI{2}{\per\hour} $. Storms were assumed uniform in space with $ \lambda = 1/24 \SI{}{\per \hour}$, and $ P_n = P_{n,e} \sim \exp(\sigma)$ with mean $ \tfrac{1}{\sigma} = \SI{5}{\milli\meter}$. The initial condition was taken to be the invariant mean $ \Xb(0) = \EXP_g(\Xb) $ of equation \eqref{Eq_EXPgX}, $ \EXP_g Q_r = \SI{750}{\liter\per\second} $. (a) Using the same data: invariant probability density function of $ Q_r $ obtained by numerically inverting the Laplace transform $ \gt_r $ in \eqref{Eq_MainResultUR}; and histogram obtained from samples of the simulation of $ Q_r(t) $ in (a) extended to $ t \in [0,200\times 24\SI{}{\hour}]$}
	}\label{Fig_simul1}
\end{figure}

\begin{remark}[Connections to the unit and geomorphic hydrographs]\label{Rem_hydrograph}
	 The function $ \vpb(t) = e^{\Mb t} $ in \eqref{Def_flow} can be appropriate called a `global hydrograph' both in the context of the classical Unit Hydrograph Theory of \cite{dooge1959general}, as well as within the more recent Geomorphologic Instantaneous Unit Hydrograph (GIUH) theory introduced by \citet{gupta1980representation}. We have added the adjective `global' to emphasize that $ \vpb $ is a matrix function that simultaneously contains the hydrographs from all streams and hillslopes in the catchment. Let $ \bs{\Theta}(t) \in \Rnn$ be the vector containing unit hydrographs at the downstream end of each link and hillslope corresponding to a unitary homogeneous precipitation event over all of $ \Gamma $. Then from \eqref{Eq_SolXt},
	\begin{equation}\label{Eq_Theta1}
	\bs{\Theta}(t) =e^{\Mb t} \left[
	\begin{array}{c}
	\mathbf{0}\\ \frac{1}{a} \Hb \ab
	\end{array}
	\right]. 
	\end{equation}
	Now, consider a GIUH model where the travel times over the hillslopes and streams corresponding to link $ e $ are taken as independent exponential random variables with respective densities
	\begin{equation}
	h_{e}(t) := H_{e} e^{-H_{e} t}, \quad k_{e}(t) := K_{e} e^{-K_{e} t}.
	\end{equation}
	Following \cite{gupta1980representation}, one obtains the representation
	\begin{equation}\label{Eq_Theta2}
	\bs{\Theta}(t) = \left[  
	\begin{array}{c}
	\left[ \sum_{e' \in \Gamma_e} \frac{a_{e'}}{a}(h_{e'}*k_{e'}*\cdots*k_e)(t)  :\; e \in \Gamma\right]^{\tr}\\ \
	\left[ \frac{a_e}{a} h_e(t) :\; e \in \Gamma\right]^{\tr}
	\end{array}
	\right]
	\end{equation}
	where $ \Gamma_e $ denotes the subnetwork with $ e $ as its outlet ($ \Gamma_r:=\Gamma$), and each summand contains convolutions following the flow path that starts at the hillslope surrounding $ e' $ and ends in stream $ e $. The fact that \eqref{Eq_Theta1} and \eqref{Eq_Theta2} coincide can be shown by differentiating the right hand side of \eqref{Eq_Theta2}.
\end{remark}

\begin{remark}[Boundary behavior]\label{Rem_boundary}
	Much of the classical treatment of PDMPs, e.g. \cite{davis1984piecewise, Rolski:1999aa}, deals with the behavior of the process at, and out of the boundary of the state space. In our case, $ \d \Rnn = \{\xb \in \Rnn: x_{i,e} = 0  \text{ for some } i=1,2, \; e \in \Gamma\}$, which is inaccessible from $ \Rnn $. Since we also refrain from considering the evolution of the process for initial conditions $ \Xb(0) \in \d\Rnn $, the boundary behavior of $ \Xb $ needs not to be specified.
\end{remark}

Denote the family of Markov transition probabilities of $ \Xb $ by
\begin{equation}\label{Def_ptx}
p(t,\xb,A) := \P_{\xb}(\Xb(t) \in A), \; A \subseteq \Rnn,
\end{equation}  
where the subscript $ \xb $ on $ \P_{\xb} $ or $ \EXP_{\xb} $ denotes probability or expectation conditioned on $ \Xb(0) = \xb $. Let $ T_t[h](\xb) := \EXP_{\xb} \left[ h(\Xb(t)) \right] $ denote the  Markov semigroup associated to $ \Xb $. Then $ T_t $ has extended infinitesimal generator given by the non-local operator
\begin{equation}\label{Def_A}
\A[h](\xb) := \grad h(\xb) \cdot \Mb \xb - \lambda h(\xb)  + \lambda \int_{\Rn} h(\xb_1,\xb_2 + \Hb (\ab \circ \yb)) f_{\Pb}(\yb) \ud \yb,
\end{equation}  
for functions $ h $ in the domain $ \Dom(\A) $ of $ \A $, which includes all continuously differentiable and bounded functions from $ \Rnn $ to $ \R $ \cite[page 366]{davis1984piecewise}. The form of the infinitesimal generator $ \A $ in \eqref{Def_A} yields a second important characterization for our process: $ \Xb $ is a conservative Feller process of the \textit{Ornstein-Uhlenbeck Type} (OUT) as described in \cite{sato1984operator}. In this case, the associated L\'evy process of $ \Xb $ has no Brownian component, its jump measure is $ \lambda e^{-\lambda t} f_{\Yb}(\yb) \ud \yb \ud t $, and the driving matrix is $ - \Mb $ which has strictly positive eigenvalues.

The transition probabilities of OUT processes have been explicitly characterized, and those of $ \Xb $ exhibit one important caveat: they are not absolutely continuous with respect to Lebesgue measure. In fact, this can be directly seen from the strong Markov property, as
\begin{equation*}
	p(t,\xb,\ud \yb) = \bs{\delta}_{\vpb(t) \xb}(\ud \yb) e^{-\lambda t}+ \\
	\int_0^t \int_{\Rnn}\!\! \lambda e^{-\lambda \tau} \fY\!\left(\zb_2 - (\vpb(\tau)\xb)_2\right) p(\tau,\zb,\ud \yb) \ud \zb \ud \tau
\end{equation*}
where $ \bs{\delta} $ denotes the Dirac delta measure, and $ \vpb $ is as in \eqref{Def_flow}. Namely, starting at $ \xb $, the atomic event $ [\Xb(t) = \vpb(t) \xb] $ has positive probability.

\citet[Theorem 3.1] {sato1984operator} provide an explicit formula for the characteristic function of $ \Xb(t) $, which we include for completeness. Their formula and most of the subsequent analysis is given here in terms of the multidimensional spatial Laplace transforms of $ p(t,\xb,\cdot) $ and $ \fY $:
\begin{align}
	\tilde p(t,\xb,\sb)  &:= \int_{\Rnn} \!\! e^{-\yb \cdot \sb} p(t,\xb,\ud \yb) , \; \sb \in \Rnn, \\
	\fYt(\sb)  &:= \int_{\Rn} \!\! e^{-\yb \cdot \sb} \fY(\yb) \ud \yb = \fPt( \Hb \ab \circ \sb ), \; \sb \in \Rn. \label{Eq_fY2tilde}
\end{align} 
The Laplace transform of the transition probability density is:
\begin{equation}\label{Eq_ptxy}
\tilde p(t,\xb,\sb) = \exp \Bigg\{ -\vpb(t)\xb \cdot \sb  
-\frac{\lambda}{H_r} \int_{e^{-H_r t}}^1 \!\!   
\frac{1-\fYt((u^{-\Mb^{\tr}/H_r} \sb)_2 )}{u} \ud u \Bigg\},
\end{equation}  
where, for $ \sb \in \Rnn $ and $ u>1 $, the vector $ (u^{-\Mb^{\tr}/H_r} \sb)_2 $ is obtained by computing the matrix exponential $ \exp\left(\frac{\log u}{H_r} \Mb^{\tr}\right) $, right-multiplying by $ \sb $, and extracting the second half of the resulting vector. A more convenient formula is given in Corollary \ref{Cor_pQr} below.

\section{Invariant density of $ \Xb $}\label{Sec_invariant}

A probability density function $ g: \Rnn \to \R_+ $ is `invariant' or `stationary' for the process $ \Xb $ if
\begin{equation}\label{Def_Invariance}
\P_g (\Xb(t) \in A)  := \int_{\Rnn} g(\xb) p(t,\xb,A) \ud \xb 
= \int_A g(\xb) \ud \xb = \P_g(\Xb(0) \in A)
\end{equation} 
for all measurable $ A \subseteq \Rnn $ and $ t > 0 $. Note that for notational convenience and without risk of confusion, in \eqref{Def_Invariance} and below we are using the symbol $ g $ to denote both, a probability density and the measure determined by it. Also, the subscript $ g $ on $ \P_g $ or $ \EXP_g $ denotes probabilities or expectations with respect to the measure $ g $, namely conditioned on $ \Xb(0)$ distributed as $ g $. For example, taking expectations throughout the stochastic differential equation \eqref{Eq_SDEX} and using invariance in the form $ \deriv{}{t} \EXP_g \Xb(t) = 0 $, we get the invariant mean streamflow and runoff
\begin{equation}\label{Eq_EXPgX}
\EXP_g \Xb = -\lambda \, \Mb^{-1} \left[
\begin{array}{c}
\mathbf{0}\\ \Hb( \ab \circ \EXP \Pb_1)
\end{array}\right]
= \lambda \left[
\begin{array}{c}
\AG^{-1}(\ab \circ \EXP \Pb_1) \\ \ab \circ \EXP \Pb_1
\end{array}\right]
.
\end{equation}   

Conditions for the existence, and the characterization of invariant distributions for OUT processes are given in \cite[Theorems 4.1-4.2]{sato1984operator}. We now apply their result to $ \Xb $.

\begin{proposition}[\cite{sato1984operator}]\label{Prop_MainResult}
	A necessary and sufficient condition for the existence of a unique invariant measure for $ \Xb $  is 
	\begin{equation}\label{Eq_SatoCond}
	\int_{\yb \in \Rn: |\yb| \geq 1} \log(|\yb|) \fY(\yb) \ud \yb < \infty.
	\end{equation}  
	If \eqref{Eq_SatoCond} holds, then the invariant distribution has Laplace transform given by
	\begin{equation}\label{Eq_MainResult}
	\tilde g(\sb) = \exp\left\{ 
	-\frac{\lambda}{H_r} \int_0^{1} \!\!   
	\frac{1-\fYt((u^{-\Mb^{\tr}/H_r} \sb)_2 )}{u} \ud u
	\right\}.
	\end{equation}  
	Moreover, $ p(t,\xb, \cdot) $ converges to $ g $ as $ t \to \infty $ for all $ \xb \in \Rnn$.
\end{proposition}

The random variable of most interest in $ \Xb(t) $ is its first entry $Q_r(t)$: the discharge at the watershed's outlet. Let $ g_r $ denote its invariant density
\begin{equation}
g_{r}(x) := \P_g(Q_r(t) \in \ud x), \quad x >0.
\end{equation}  
Crucially, the Laplace transform of $ g_r $ is easily obtained from that of $ g $ as
\begin{equation}
\gt_r(s) = \gt([s,0,\dots,0]^{\tr}).
\end{equation}  
We then have the following Corollary.

\begin{corollary}\label{Cor_pQr}
	The discharge $ Q_e $ has limiting invariant density $ g_e $ given by
	\begin{equation}\label{Eq_MainResultQe}
	\tilde g_e(s) = \exp\left\{ 
	-\frac{\lambda}{H_r} \int_0^{1} \!\!   
	\frac{1-\fPt( \Hb \ab \circ \mb_e(u) \, s )}{u} \ud u
	\right\}, \quad e \in \Gamma, \; s>0,
	\end{equation}  
	where $ \mb_e(u) $ is the second half  of the column of $ u^{-\Mb^{\tr}/H_r}  $ corresponding to $ Q_e $. 
\end{corollary}

Once $ \tilde g $ has been computed, the Laplace transform of the invariant distribution of any of the components of $ \Xb $ can be numerically inverted to get the corresponding approximate density function. See Figures \ref{Fig_simul1}b and \ref{Fig_fullOrder4}b for example. For all computations reported here, we use the classical algorithm by \cite{zakian1969numerical}.

\begin{figure}[h] \centering
	\hspace*{-0.5cm} \includegraphics[scale=1]{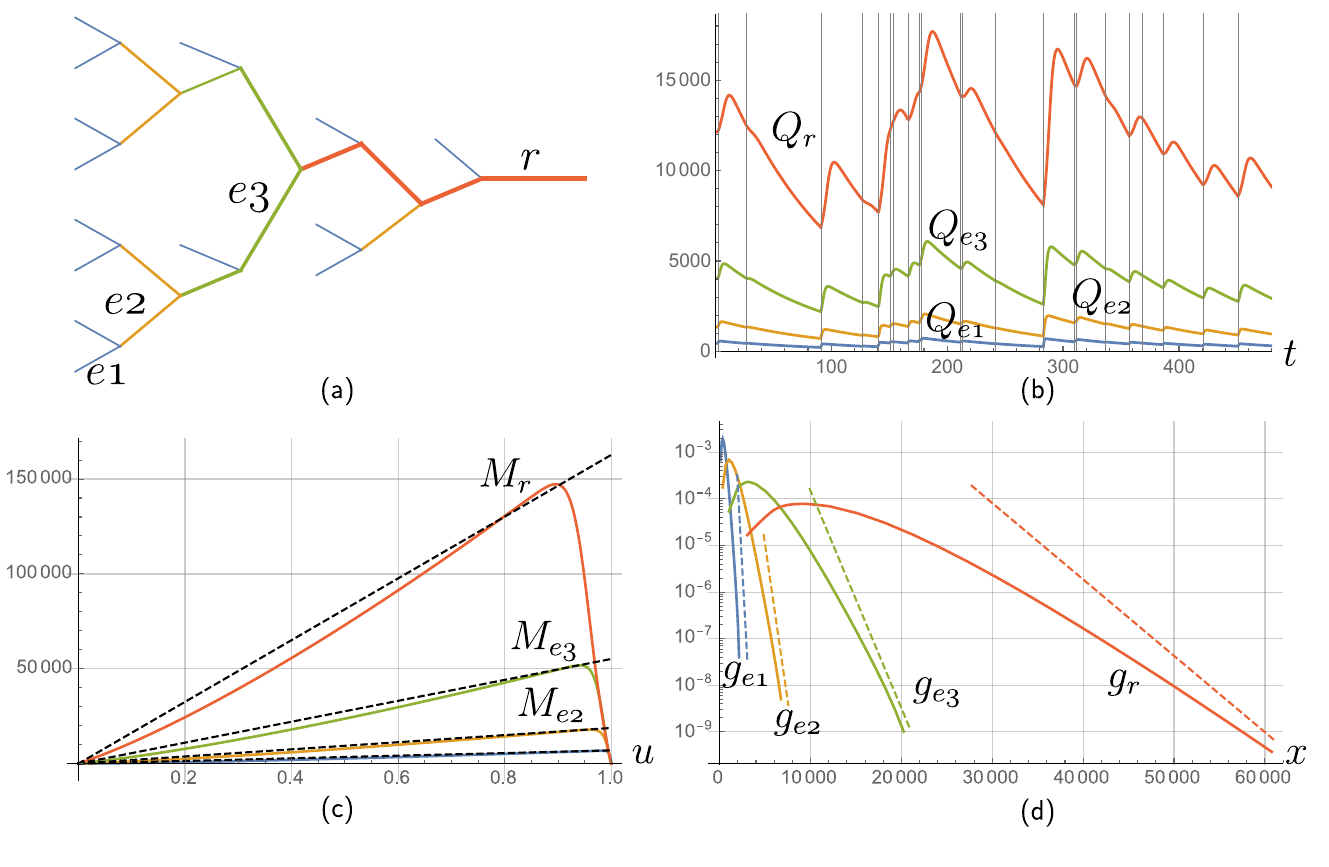}
	\caption{\small{
			Simulation and invariant densities for a watershed of fourth order. The precipitation field satisfies $ \lambda = \tfrac{1}{24} \SI{}{\per \hour}$, $ P_n = P_{n,e} \sim \exp(\sigma)$ with  $ \tfrac{1}{\sigma} = \SI{5}{\milli \meter}$. All hillslope areas were assumed equal to $ a_e = \SI{0.6}{\kilo\meter^2}$. Values of $ K_e, H_e $ were randomly generated as $ K_e = \epsilon_K K $, $ H_e = \epsilon_H H $ with $ H/K = 0.008$, $ \epsilon_K, \eps_H \sim \text{Unif}[0.5,1.5] $ (a) Structure of the river network with four selected stream links of increasing Horton order. (b) Simulated discharge $ Q_e(t) $ in \SI{}{\liter \per \second} at selected streams for $ t \in [0,20 \times 24\SI{}{\hour}] $. Vertical lines mark storm events. (c) Plots of $ M_e(u) :=  \Hb \ab \cdot \mb_e(u)$ compared with the linear approximation $ \Hb \ab \cdot \mb^0_e(u) $ in \eqref{Eq_m0}. (d) Logplots of the invariant densities $ g_e(x) $ for selected streams and $ x \in [0.25,6]\times \EXP_g Q_e $ \SI{}{\liter \per \second}. The straight dashed lines have the slope predicted in Proposition \ref{Prop_Tails}.
		}\label{Fig_fullOrder4}}
\end{figure}

\begin{remark}[Spatially uniform rainfall]\label{Rem_UnifRain}
	The case where $ P_{n,e} = P_n $ for all $ n \geq 1 $ and $ e \in \Gamma $, represents a scenario where on every storm, all hillslopes receive the same random amount $ P_n $ of rainfall. In this case it suffices to consider an i.i.d. sequence $ \{P_n: n \geq 1\} $ of rainfall depths with common probability density $ f_P $ instead of the joint density $ \fP $. The description of the process is obtained by replacing $ \Hb(\ab \circ \Pb_n) $ by $ \Hb \ab P_n $ in \eqref{Def_Yb}. Similarly, the integrands in \eqref{Eq_ptxy}, \eqref{Eq_MainResult} and \eqref{Eq_MainResultQe} must be modified by replacing Hadamard products with dot products: using expression \eqref{Eq_fY2tilde} for $ \fYt $, gives
	\begin{equation}\label{Eq_MainResultUR}
	\tilde g_e(s) = \exp\left\{ 
	-\frac{\lambda}{H_r} \int_0^{1} \!\!   
	\frac{1-\tilde f_P(\Hb \ab \cdot \mb_e(u) \, s )}{u} \ud u
	\right\}.
	\end{equation}
\end{remark}

\begin{remark}[The functions $ \mb_e $: calculation and approximation]
	The functions $ \mb_e $, $ e\in \Gamma $ in \eqref{Eq_MainResultQe} and \eqref{Eq_MainResultUR}  encapsulate the role played by the network geomorphology on the asymptotic distribution of discharge.  Some remarks are in order. Note first that the value of $ H_r $ in \eqref{Eq_MainResult} and \eqref{Eq_MainResultQe} can be changed to any $ H_e $ (or any other positive rate) by making the change of variables $ u \mapsto u^{H_e/H_r} $ in the integral. Secondly, if we denote the lower left block matrix of $ u^{-\Mb^{\tr}/H_r} $ by $ \mb(u) \in \R^{n_{\Gamma} \times n_{\Gamma}}$, then $ \mb_e(u) $ is the column of $ \mb(u) $ corresponding to link $ e $. By \eqref{Def_M} the matrix $ \mb(u) $ is given by
	\begin{equation}\label{Eq_mSeries}
	\mb(u) = \sum_{n=0}^{\infty}- \frac{\log(u)^n}{H_r^n n!}
	\left\{\Kb \left[\Ib - \AG \Kb \Hb \right]^{-1} \left[\Ib - (\AG \Kb \Hb^{-1})^n\right] \Hb^{n-1} \right\}^{\tr}.
	\end{equation}
	In particular, in the spatially uniform case where $ K_e = K_r $  and $ H_e = H_r $ for all $ e \in \Gamma $, one can write
	\begin{equation}\label{Eq_mdiag}
	\mb(u) = \left\{ \left[ \AG - \beta \Ib \right]^{-1} \left[u \Ib - u ^{\frac{1}{\beta} \AG }\right]\right\}^{\tr}, \quad \beta:= \frac{H_r}{K_r}.
	\end{equation}
	If $  H_e = H_r \ll K_r = K_e  $ for all $ e \in \Gamma $, the following linear approximation holds (see Figure \ref{Fig_fullOrder4}c)
	\begin{equation}\label{Eq_m0}
	\mb(u) \approx \lim_{\beta \to 0} \mb(u) = (\AG^{-1})^{\tr}u =: \mb^0(u)
	\end{equation} 
	As illustrated in Figure \ref{Fig_heterogeneous}, expression \eqref{Eq_mdiag} for homogeneous self-similar networks seems to yield self-similar invariant distributions for the discharge. Deviations from homogeneity in the parameters $ H_e, K_e $ show much richer behavior.
\end{remark}

\begin{figure}[h] \centering
	\hspace*{-0.5cm} \includegraphics[scale=1.05]{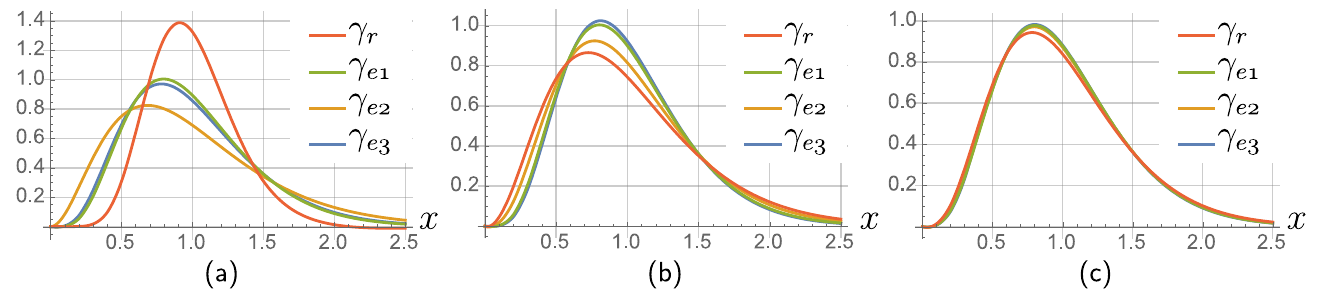}
	\caption{\small{
			(The effect of heterogeneity) Scaled densities $ \gamma_e $ of the normalized discharge $ Q_e/ \EXP_g Q_e $ for selected links of the river network and precipitation regime shown in Figure \ref{Fig_fullOrder4}.  The values of $ H_e $ and $ K_e $ were assigned as in Figure \ref{Fig_fullOrder4} but with random multipliers $ \epsilon $ of decreasing variance: a) $ \epsilon_K, \eps_H \sim \text{Unif}[0.25,1.75] $, b) $ \epsilon_K, \eps_H \sim \text{Unif}[0.5,1.5] $, c) $ \epsilon_K, \eps_H \sim \text{Unif}[0.9,1.1] $.
		}\label{Fig_heterogeneous}}
\end{figure}

\section{Moments and extreme events of $ Q_r $}\label{Sec_MomsExtreme}

In this section we exploit the properties of the Laplace transform of the density $ g_{r} $ to derive important characteristics of the invariant distribution of $ \Xb $ and of $ Q_r $. For simplicity, we restrict our attention to the case of uniform rainfall as described in Remark \ref{Rem_UnifRain}.

\subsection{Calculation of moments}\label{Sec_Moments}

Denote the $ n $-th invariant moment of $ Q_r $ by
\begin{equation}\label{Def_MnQr}
\Momr{n} := \EXP_g Q_r^n = (-1)^n \deriv[n]{\gt_r}{s}(0), \; n=1,2,\dots
\end{equation}  
If we denote the function inside the exponential in $ \eqref{Eq_MainResult} $ by
\begin{equation}
h(s) := - \frac{\lambda}{H_r}
\int_0^{1}\!\! \frac{1 - \tilde f_P(\Hb \ab \cdot \mb_r(u) s )}{u} \ud u,
\end{equation}  
then Faa di Bruno's formula for the $ n $-th derivative of $ e^{h(s)} $ gives
\begin{equation}\label{Eq_FaaDB}
\deriv[n]{\tilde g_r}{s} = \tilde g_r  \sum_{k=1}^{n} B_{n,k}\left(\left\{  \deriv[i]{h}{s} \right\}_{i=1}^{n-k+1} \right)
\end{equation}  
where $ B_{n,k} $ is the Bell polynomial
\begin{equation}\label{Def_BellPolys}
B_{n,k}(\{x_i\}_{i=1}^{n-k+1}) = n! \sum_{\boldsymbol{j} \in I(n,k)} 
\prod_{i=1}^{n-k+1} \dfrac{1}{j_i!}\left(\dfrac{x_i}{i}\right)^{j_i}.
\end{equation}  
Here, $ \bs{j}\ \in I(n,k) $ denotes that the sum is taken over all vectors $ \bs{j} = \{j_1,\dots,j_{n-k+1} \} $ of indices such that:
\begin{equation}\label{Def_BellIndices}
\sum_{i=1}^{n-k+1} j_i = k, \quad \sum_{i=1}^{n-k+1} i j_i = n.
\end{equation}  
Letting $ s \downarrow 0 $ in \eqref{Eq_FaaDB}, and writing $ \MomP{n} $ for the $ n $-th moment of $ P_1 $, one arrives at the following useful expression (see Figure \ref{Fig_moments} for a numerical example).

\begin{proposition}\label{Prop_Moments}
	The $ n $-th invariant moment of $ Q_r $ is
	\begin{equation}\label{Eq_Mom}
	\Momr{n} = (a K_r)^n \sum_{k=1}^{n} \left(\frac{\lambda}{H_r}\right)^k \!\! B_{n,k}\left( 
	\left\{ \MomP{i} c_i
	\right\}_{i=1}^{ n-k+1}
	\right )
	\end{equation}  
	where the coefficients $ c_i $ are given by
	\begin{equation}\label{Def_cn}
	c_\alpha := \int_0^{1} \dfrac{[\bs \beta \tilde\ab \cdot \mb_r(u)]^\alpha}{u} \ud u, \;\; \alpha>0, 
	\end{equation}  
	with $ \bs \beta := \left[\tfrac{H_e}{K_r}: e \in \Gamma \right]^{\tr} $ and $ \tilde \ab := \tfrac{1}{a} \ab$. 
\end{proposition}

\begin{remark}\label{Rem_cns}
	There are two important implications of the Proposition 
	\ref{Prop_Moments}. First, that under the invariant distribution, the discharge will have exactly as many moments as each $ P_n $. Secondly, the non-dimensional constants $ \{ c_n : n=1,2, \dots \} $  constitute a set of  parameters, depending only on the geomorphology of the watershed, that completely determine the invariant distribution of the discharge $ Q_r $. 
\end{remark}

\begin{figure}[h] \centering
	\includegraphics[scale=1.1]{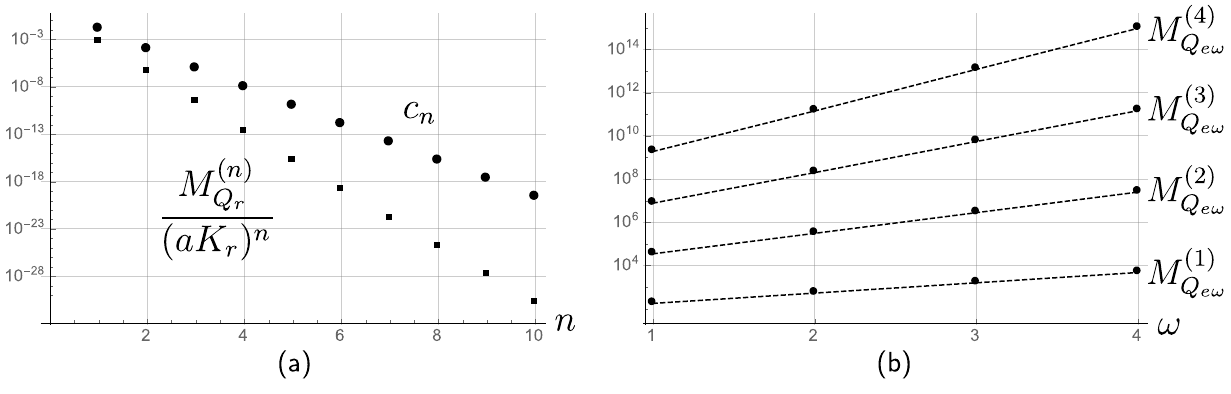} 
	\caption{\small{
			For the example of figure \ref{Fig_fullOrder4}: (a) logarithmic plots of the sequence of $ c_n $ (unit-less), and scaled moments $ \Momr{n}/(a K_r)^n $ in units of $ \text{m}^n $, $ n=1,\dots,10 $. (b) First four invariant moments of $ Q_e $ for the selected links in Figure \ref{Fig_fullOrder4} which have consecutive Horton order $ \omega = 1,2,3,4 $, $ \edge{4}:=r $. }
	}\label{Fig_moments}
\end{figure}

Direct integration of \eqref{Eq_Mom} shows that
\begin{equation}\label{Eq_EXPqQr}
c_1 = \dfrac{H_r}{K_r}, \;\; \EXP_g Q_r = a \lambda \EXP(P_1)
\end{equation}  
which coincides with the expression for $ \EXP_g \Xb $ in \eqref{Eq_EXPgX}, and makes explicit the fact that under the invariant distribution, the discharge is in a state of average equilibrium with the precipitation.

\subsection{Asymptotics of extreme events}\label{Sec_Tails}

We now turn our attention to the asymptotic behavior of the probabilities of extreme events of discharge, namely $ \P_g (Q_e > x) $  as $ x \to +\infty $. We are particularly interested on how this decay relates to both, the geomorphological properties of the network, and the probablity density $ f_P $ of rainfall. First of all, as noted in Remark \ref{Rem_cns}, if the distribution of $ P_1 $ has $ n $ finite moments, then it follows from Proposition \ref{Prop_Moments} that $ Q_r $ will also have exactly $ n $ finite moments. In that case, general theory \citep[see for example][Excercise 3.2.5]{chung2001course} gives 
\begin{equation}
	\lim_{x \to \infty } x^n \P_g(Q_r > x)  = 0.
\end{equation}

We now give precise results for the asymptotics of $ \P_g (Q_e > x) $ in two contrasting types of distributions of $ P_1 $: a heavy-tailed distribution with no moments, and the exponential distribution. In both cases, we proof that the invariant distribution of $ Q_e $ preserves the general asymptotic behavior as that of $ P_1 $. See Figure \ref{Fig_fullOrder4}d.

\begin{proposition}\label{Prop_TailsPot}
	Suppose $ P_1 \sim \text{Pareto}(\alpha, k) $ for some $k >0$ and $ 0<\alpha<1$. Then
	\begin{equation}\label{Eq_ExpTailsPot}
	\P_g(Q_r > x) \sim \frac{\lambda (k a K_r)^\alpha}{H_r} c_\alpha \, x^{-\alpha } \text{ as } x \to \infty 
	\end{equation}
	where $ c_\alpha $ is given by \eqref{Def_cn}.
\end{proposition}
 \begin{proof}
 	
 \end{proof}
	In this case $ f_P(x) \sim x^{-1-\alpha}$ as $ x \to \infty $. Equation \eqref{Eq_Mom} along with the expansion of $ \fPt $ in Taylor series, yields 
	\[  \gt_r \sim \exp(-C s^\alpha) \text{ where } C=\frac{\lambda}{H_r}k^\alpha \Gamma(1-\alpha) \int_0^1 \frac{(\Hb \ab \cdot \mb_r(u))^\alpha}{u} \ud u  \]
	as $ s \downarrow 0 $. Denote $ \Psi(x):=\P(Q_r > x) $. Then $ \tilde{\Psi}(s) =\frac{1}{s}(1-\gt(s))\sim C s^{\alpha - 1}  $ as $ s \downarrow 0 $. The Karamata Tuberian theorem gives that $ \int_0^x \Psi(y) \ud y \sim \frac{C}{\Gamma(2-\alpha)}x^{1-\alpha} $ as $ x \to \infty $. Differentiation yields the desired asymptotics for $ \Psi(x) $. See \citet{bingham1989regular}, Theorems 1.7.1, 1.7.2.

\begin{proposition}\label{Prop_Tails}
	Suppose $ P_1 \sim \exp(\sigma) $ for $ \sigma >0 $. Let $ M_r^* =\max\{ \Hb \ab \cdot \mb_r(u): u \in [0,1] \} $, then
	\begin{equation}\label{Eq_ExpTails}
	\lim_{x \to \infty} \dfrac{1}{x} \log \P_g(Q_r > x) = -\dfrac{\sigma}{M_r^*}. 
	\end{equation}  
\end{proposition}
\begin{proof}
	The tails of $ f_P(y) $ decay exponentially which is manifested in $ \tilde f_P(s) = \sigma/(s+\sigma) $ as a pole at $ s = -\sigma $. We now show the same a pole also exists for $ \gt_r $. Note first that for $ u \in [0,1] $, the matrix $ -\frac{\log u}{H_r}\Mb^{\tr} $ has non-negative off-diagonal entries, and therefore every entry of $ u^{-\Mb^{\tr}/H_r} $ is non-negative. This implies that $M_r(u):= \Hb\ab \cdot \mb_r(u) $, $ u\in [0,1] $ is a non-negative function with $ M_r(0) = M_r(1) = 0 $. See Figure \ref{Fig_fullOrder4}d. Moreover by \eqref{Eq_mSeries}, $ M_r $ is a bounded and differentiable function of $ u $ with a positive maximum at $ u =u^* \in (0,1) $ where $ M_r'(u^*) = 0 $. The convergence of $ \int_0^1 \tfrac{1}{u} (1-\fPt(M_r(u)s)) \ud u $ for $ s>0 $ is guaranteed by Proposition \ref{Prop_MainResult}. For $ -\tfrac{\sigma}{M_r^*} <s \leq 0 $ the following estimate holds
	\[  \int_0^1 \tfrac{1}{u} (1-\tilde f_P(M_r(u)s)) \ud u \geq \dfrac{s M_r^*}{u^*} \log \left(1+\dfrac{u^*}{\sigma + s M_r^*}\right)  \]
	and therefore $ \gt_r(s) \to \infty $ as $ s \downarrow -\sigma/M_r^* $. The asymptotic formula \eqref{Eq_ExpTails} now follows from \citet[Theorem 3]{nakagawa2005tail}
\end{proof}

\section{Conclusion and outlook}\label{Sec_Discussion}

In this work we have presented the detailed mathematical solution for the equations \eqref{Eq_ODEsQ}-\eqref{Eq_ODEsR} of mass and momentum balance in an arbitrary watershed at the hillslope scale. The solution covers the deterministic case through the global hydrograph map $ \vpb $ in \eqref{Def_flow}, as well as the case of rainfall given by a Poisson point process. In its most basic form, our main result gives an approximation for the distribution of runoff and streamflow within a watershed given the geometry of the river network and a set of hillslope-scale physical parameters.

An immediate consequence of Proposition \ref{Prop_MainResult} is that $ \Xb $ is an `ergodic' process (see \cite[Section 20]{Kallenberg} for Feller ergodic  processes, or \cite{Costa:2008aa} for the specific case of ergodic PDMPs). It follows in particular, that 
\begin{equation}\label{Eq_ergodic}
\lim_{t \to \infty} \frac{1}{t} \int_0^t f(\Xb(s)) \ud s = \lim_{t \to \infty} T_t[f](\xb) =  \int_{\Rnn} f(\xb)  g(\xb) \ud \xb
\end{equation}  
$ g $-almost surely for all suitable $ f: \Rnn \to \R $ and any initial condition $ \Xb(0) = \xb \in \Rnn$. If the statistical properties of the precipitation field do not change for a sufficiently long period of time, our model predicts that the watershed will attain a statistical invariant regime determined by $ g $.

On a final comment, and as thoroughly explained in \cite{Gupta:2007he}, it should be noted that solutions to \eqref{Eq_ODEsQ}-\eqref{Eq_ODEsR} can be used to study the phenomenon of statistical scaling in watersheds: power law relationships between the frequency and magnitude of streamflow, and physical parameters of the watershed. The mathematical framework developed here provides a concrete technical base to undertake such scaling analysis as illustrated in Figures \ref{Fig_heterogeneous} and \ref{Fig_moments}b. We thought however that this present note should restrict its focus to the mathematical details of the solution to the model. Results on the scaling properties of the invariant distribution of $ \Xb $, its moments and tail behavior, will be included in a separate note.

\section*{Acknowledgments}
	All relevant numerical calculations and examples included in this manuscript were performed in Mathematica\texttrademark and an interactive notebook is available upon request to the corresponding author. 
	This research was funded by The British Council, through a Researchers Link grant to The University of Liverpool and Universidad Nacional de Colombia. The project also benefited from funding from the European Union's Seventh Framework Programme for research, technological development and demonstration under grant agreement number 318984-RARE.
	The authors would also like to thank several people whose comments and suggestions made this work possible. Including Professors Oscar Mesa, Germ\'an Poveda and Carlos Hoyos  (Universidad Nacional de Colombia), Alfred Muller (Siegen University), Junyi Guo (Nankai University), Tomasz Kozubowski (University of Nevada), Manuel Correa (Integral S.A., Colombia), Vijay Gupta (University of Colorado, Boulder), Edward Waymire (Oregon State University) and Jesus G\'omez (New Mexico Tech).


\bibliographystyle{abbrvnat}
\bibliography{DMP_Floods}

\end{document}